\documentclass[11pt]{amsart}

\newif\ifdraft
\draftfalse

\usepackage{amsmath,amsthm,verbatim,amssymb,amsfonts,amscd, graphicx}
\usepackage{graphics}
\usepackage{theoremref}
\usepackage{mathtools}
\usepackage[alphabetic,initials]{amsrefs}

\usepackage{tikz}

\usepackage{tikz-cd}

\usepackage{amsmath, amscd}%
\usepackage{amsfonts}%
\usepackage{amssymb}

\usepackage{enumitem}

\usepackage{graphicx}
\usepackage[pdftex=true, plainpages=false]{hyperref}

\numberwithin{equation}{subsection}

\renewcommand{\tilde}{\widetilde}
\newtheorem{thm}{Theorem}[section]

\newtheorem{cor}[thm]{Corollary}

\newtheorem{lemma}[thm]{Lemma}
\newtheorem{prop}[thm]{Proposition}

\theoremstyle{remark}

\newtheorem*{rmk}{Remark}
\theoremstyle{definition}
\newtheorem{de}[thm]{Definition}

\newcommand{\R}{{\mathbb R}}

\newcommand{\bbC}{{\mathbb C}}

\newcommand{\Q}{{\mathbb Q}}
\newcommand{\Z}{{\mathbb Z}}
\newcommand{\N}{{\mathbb N}}

\newcommand{\VV}{{\mathbb{V}}}

\DeclareMathOperator{\Res}{Res}

\DeclareMathOperator{\supp}{supp}

\DeclareMathOperator{\gr}{Gr}
\DeclareMathOperator{\DR}{DR}

\newcommand{\cB}{\mathcal{B}}
\newcommand{\cC}{\mathcal{C}}

\newcommand{\cE}{\mathcal{E}}

\newcommand{\cH}{\mathcal{H}}
\newcommand{\cI}{\mathcal{I}}

\newcommand{\cL}{\mathcal{L}}
\newcommand{\cM}{\mathcal{M}}

\newcommand{\cO}{\mathcal{O}}

\newcommand{\cT}{\mathcal{T}}

\newcommand{\cV}{\mathcal{V}}

\newcommand{\theoremref}[1]{\hyperref[#1]{Theorem~\ref*{#1}}}
\newcommand{\lemmaref}[1]{\hyperref[#1]{Lemma~\ref*{#1}}}
\newcommand{\definitionref}[1]{\hyperref[#1]{Definition~\ref*{#1}}}
\newcommand{\propositionref}[1]{\hyperref[#1]{Proposition~\ref*{#1}}}
\newcommand{\conjectureref}[1]{\hyperref[#1]{Conjecture~\ref*{#1}}}
\newcommand{\corollaryref}[1]{\hyperref[#1]{Corollary~\ref*{#1}}}
\newcommand{\exampleref}[1]{\hyperref[#1]{Example~\ref*{#1}}}

\newcommand{\be}{\begin{enumerate}}
\newcommand{\ee}{\end{enumerate}}
\newcommand{\bt}{\begin{thm}}
\newcommand{\et}{\end{thm}}
\newcommand{\bde}{\begin{de}}

\newcommand{\ede}{\end{de}}
\newcommand{\bc}{\begin{cor}}
\newcommand{\ec}{\end{cor}}
\newcommand{\blm}{\begin{lemma}}
\newcommand{\elm}{\end{lemma}}
\newcommand{\bp}{\begin{proof}}
\newcommand{\ep}{\end{proof}}
\newcommand{\beq}{\begin{equation*}\label{xx}}
\newcommand{\eeq}{\end{equation*}}

\renewcommand{\nu}{{\mathcal{V}}}

\newcommand{\lee}{\leqslant}
\newcommand{\gee}{\geqslant}

\newcommand{\norm}[1]{\left\lVert#1\right\rVert}

\renewcommand{\ref}[1]{\hyperref[#1]{\ref*{#1}}}

\renewcommand{\d}[1]{{#1^\bullet}}

\newcommand{\Dmod}{\mathfrak{D}}

\newcommand{\rarr}{\longrightarrow}

\newcommand{\abs}[1]{\lvert #1 \rvert}
\newcommand{\floor}[1]{\lfloor #1 \rfloor}
\newcommand{\ceil}[1]{\lceil #1 \rceil}

\makeatletter

\newcommand{\Rmnum}[1]{\expandafter\@slowromancap\romannumeral #1@}
\makeatother

\begin{document}
\title{Vanishing and injectivity for $\R$-Hodge modules and $\R$-divisors}

\author{Lei Wu}
\address{Department of Mathematics, University of Utah,
155 S 1400 E, Salt Lake City, UT 84112, USA}
\email{\tt lwu@math.utah.edu}
\date{}

\begin{abstract}
We prove the injectivity and vanishing theorem for $\R$-Hodge modules and $\R$-divisors over projective varieties, extending the results for rational Hodge modules and integral divisors in \cite{Wu15}. In particular, the injectivity generalizes the fundamental injectivity of Esnault-Viehweg for normal crossing $\Q$-divisors, while the vanishing generalizes Kawamata-Viehweg vanishing for $\Q$-divisors. As a main application, we also deduce a Fujita-type freeness result for $\R$-Hodge modules in the normal crossing case.
\end{abstract}
\maketitle
\section{Introduction and Main Theorems}
A Kawamata-Viehweg-type vanishing for rational Hodge modules has been proved in \cite{Suh15} and \cite{Wu15}. It generalizes Kodaira-Saito vanishing and many other vanishing results of its geometric realizations in the rational Hodge-module case for big and nef line bundles. In \cite{Wu15}, the vanishing has been further improved to an injectivity result in a natural way using the degeneration of Hodge-to-de Rham spectral sequence following the idea of Esnault-Viehweg. The method in $loc.$ $cit.$ gives clues about the existence of both the injectivity and the vanishing for rational Hodge modules and $\Q$-divisors; see \cite[Theorem 7.5 and 7.6]{Wu15}.    
In this paper, we make generalization of both the injectivity and the vanishing in the real case for real Hodge modules and $\R$-divisors over projective manifolds, and hence in the rational case for rational Hodge modules and $\Q$-divisors as well. 

Inspired by Esnault-Viehweg's beaustiful ideas \cite{EV}, we use the following key observation: 
{\bf Hodge modules twisted by rank-one unitary representations should be considered for the purpose of injectivity and vanishing.}  

The other key technical tools that we are using are V-filtrations (the $\R$-indexed ones) of the Deligne meromorphic extensions, which contains exactly the same information as the multi-indexed Deligne lattices (see \S2.1 and Proposition \ref{prop:vfDe}), real Hodge modules in the normal crossing case in the sense of Saito \cite{Sai90a} and the theory of tame harmonic bundles in the sense of Simpson and Mochizuki \cite{Sim90, Moc06, Moc09}. 
   
\subsection{Main results}Now we state the main results. We consider a complex manifold $X$ with a simple normal crossing divisor $D=\sum_{i\in I} D_i$ (that is, locally $D=(z_1\cdots z_r=0)$ with coordinates $(z_1,\dots,z_n)$ and each $D_i$ is smooth). Suppose that $V=(\cV, F_\bullet, \VV_\R)$ is a polarizable real variation of Hodge structures (VHS) defined over $U=X\setminus D$. 

For every $\R$-divisor $B=\sum_{i\in I} t_i D_i$ supported on $D$, we denote by $\cV^{\lee B}$ (resp. $\cV^{\gee B}$) the upper (resp. lower) canonical Deligne lattice of index $B$. From construction (see \S2.1), $\cV^{\lee B}$ (resp. $\cV^{\geqslant B}$) are characterized by 
\[\text{eigenvalues of }\Res_{D_i}\nabla \in (t_i-1, t_i] (\text{resp. } [t_i, t_i+1))\]
for every $i$, where $\Res_{D_i}\nabla$ the residue of the flat connection on $\cV$ along $D_i$. Then $\cV^{\lee B}$ (resp. $\cV^{\gee B}$) has an induced filtration defined by 
$$F^{\lee B}_\bullet=\cV^{\lee B}\cap j_* F_\bullet \textup{(resp. } F^{\gee B}_\bullet=\cV^{\gee B}\cap j_* F_\bullet);$$ they are indeed subbundle filtrations by Lemma \ref{lem:main}; that is, every Deligne lattice of a polarizable VHS is a filtered lattice with the Hodge filtration. Denote by $S^{\lee B}(V)$ (resp. $S^{\gee B}(V)$) the first term in the filtration. When $V=(\cO_U, F_\bullet, \R_U)$ is the trivial VHS, one easily observes 
$$S^{\lee B}(V)=S^{\gee B}(V)=\cO_X(-\lfloor B\rfloor).$$

All $\cV^{\lee B}$ (resp. $\cV^{\gee B}$) together define a locally discrete discreasing and upper (resp. lower) semicontinuous multi-indexed filtration on the Deligne meromorphic extension $\cV(*D)$; see Corollary \ref{cor:semiconfil}. Local discreteness and semicontinuity of the filtration enable us to give the following definition of jumping divisors for the upper (resp. lower) canonical Deligne lattice analogous to jumping numbers for multiplier ideals. 
\begin{de}
We define $\R$-divisors 
$D_u=D_u(\cV)=\inf\{B| \cV^{\lee B}=\cV^{\lee 0}\}$ and $D_l=D_l(\cV)=\sup\{B| \cV^{\gee B}=\cV^{\gee 0}\}$.
\end{de}
For instance, when $\cV=\cO_U$, the Deligne  lattice jumps integrally and we have 
\begin{equation}\label{eq:trjump}
D_l=-D_u=D.
\end{equation}
\begin{thm}\label{thm:maininj}
With notations as above, suppose that $X$ is projective. Let $\cL$ be a line bundle so that $\cL^N\simeq \cO_X(D')$ for an effective divisor $D'$ supported on $D$ and $N\in \N$.  Then we have
\begin{enumerate}
\item{if $\frac{1}{N}D'\le D_l$, then for every effective divisor $E$ supported on $\supp(D')$ the natural map 
\[H^i(X, S^{\lee0}(V)\otimes\omega_X\otimes \cL)\to H^i(X, S^{\lee0}(V)\otimes\omega_X\otimes \cL(E))\]
is injective;}\label{item:inj1}
\item{if $\frac{1}{N}D'+D_u<0$, then for every effective divisor $E$ supported on $D$ the natural map 
\[H^i(X, S^{\gee -D}(V)\otimes\omega_X\otimes \cL^{-1})\to H^i(X, S^{\gee -D}(V)\otimes\omega_X\otimes \cL^{-1}(E))\]
is injective.}\label{item:inj2}
\end{enumerate}
\end{thm}

After taking $V$ the trivial VHS, in particular, the above injectivity gives a different proof of the fundamental injectivity of Esnault and Viehweg \cite[\S5]{EV}, without using positive characteristic methods of Deligne and Illusie.
\begin{cor}[Esnault and Viehweg]\label{cor:injtrivial}
With notations as above, suppose that $X$ is projective. Let $\cL$ be a line bundle so that $\cL^N\simeq \cO_X(D')$ for an effective divisor $D'$ supported on $D$ and $N\in \N$.  Then we have
\begin{enumerate}
\item{if $\frac{1}{N}D'\le D$, then for every effective divisor $E$ supported on $\supp(D')$ the natural map 
\[H^i(X, \omega_X\otimes \cL)\to H^i(X, \omega_X\otimes \cL(E))\]
is injective;}
\item{if $\frac{1}{N}D'<D$, then for every effective divisor $E$ supported on $D$ the natural map 
\[H^i(X, \omega_X\otimes \cL^{-1}(D))\to H^i(X, \omega_X\otimes \cL^{-1}(D+E))\]
is injective.}
\end{enumerate}
\end{cor}

Theorem \ref{thm:maininj}  \eqref{item:inj1} implies the following injectivity for nef and big $\R$-divisors, which fully generalizes the injectivity \cite[Thm. 1.4]{Wu15}. 
\begin{thm}\label{thm:inj}
Assume that $X$ is a projective manifold. Let $L$ be a divisor on $X$. If $L- B$ is nef and big for some $\R$-divisor $B$ supported on $D$, then for every effective divisor $E$ the natural morphism 
\[H^i(X, S^{\lee B}(V)\otimes\omega_X(L))\to H^i(X, S^{\lee B}(V)\otimes\omega_X(L+E))\]
is injective.
\end{thm} 
When $V$ is the trival VHS, since $S^{\lee B}(V)=\cO_X(-\lfloor B\rfloor)$, the above theorem specializes to the injectivity theorem of Koll\'ar and Esnault-Viehweg; see for instance \cite[Cor. 5.12(b)]{EV}. More generally, when $V$ is the $(d+i)$-th geometric VHS given by a family $f$ (that is the VHS parametrizing the $(d+i)$-th cohomology of smooth fibers), we know that from construction when $B$ is effective
$$S^{\lee B}(V)\subseteq S^{\lee 0}(V)=R^if_*\omega_{Y/X},$$ where $f: Y\to X$ is a surjective projective morphism branched along $D$; cf. \cite{Kol86b} and \cite{SaitoKC}. In this case, after taking $B=0$, it particularly implies Koll\'ar's injectivity for higher direct images of dualizing sheaves \cite[Thm. 2.2]{Kol86a}. In other words, in the geometric case it is also an $\R$-divisor generalization of Koll\'ar's injectivity for projective families branched along normal crossing divisors.  

Taking $E$ ample, the following Kawamata-Viehweg type vanishing follows immediately, which makes full generalization of the vanishing in \cite{Wu15} and \cite{Suh15}; see also \cite[Thm.1.1.6]{WuThesis} for $\Q$-Hodge modules from an alternative point of view using Kawamata coverings. 
 \begin{thm}\label{thm:mainvan}
With notations as in Theorem \ref{thm:inj}, if $L- B$ is nef and big for some $\R$-divisor $B$ supported on $D$, then 
\[H^i(X, S^{\lee B}(V)\otimes\omega_X(L))=0\]
for $i>0$.
\end{thm} 

Similar to injectivity, when $V$ is trivial, the above vanishing theorem specially gives Kawamata-Viehweg vanishing for $\R$-divisors \cite[Cor. 5.12(b)]{EV}, while in the relative case, it generalizes Koll\'ar vanishing  \cite[Thm. 2.1(iii)]{Kol86a} for projective families branched along normal crossing divisors. 

We also obtain a relative vanishing under birational morphisms using the strictness of the direct images of mixed Hodge modules, which is a natural generalization of relative Kawamata-Viehweg vanishing in the log canonical case; see Theorem \ref{prop:locvanlc}.

As an application of Theorem \ref{thm:mainvan}, we deduce the following Fujita-type global generation result. 
\begin{thm}[=Theorem \ref{thm:fgg}]
Let $X$ be a smooth projective variety of dimension $n$ with $D$ a simple normal crossing divisor and $V$ a polarizable real VHS on $X\setminus D$, and let $L$ be an ample divisor on $X$ and a point $x\in X$. Assume that for every klt pair $(X, B_0)$, there exists an effective $\Q$-divisor $B$ on $X$ satisfying the following conditions:
\begin{enumerate} [label=(\roman*)]
\item{$B\equiv \lambda L$ (numerical equivalence) for some $0<\lambda<1$;}
\item{$(X, B+B_0)$ is log canonical at $x$;}
\item{$\{x\}$ is a log canonical center of $(X, B+B_0)$.}
\end{enumerate}
Then the natural morphism 
\[H^0(X, S^{\lee 0}(V)\otimes \omega_X(L))\rarr S^{\lee 0}(V)\otimes \omega_X(L)|_{\{x\}}\]
is surjective. 
\end{thm}
It is known from \cite{EL}, \cite{Kaw97} and recently \cite{YZ} that the assumption for $kL$ always holds for $k\ge n+1$ and $n\le 5$, from \cite{ASiu} for $k\ge$${n+1}\choose 2$ and $n$ arbitrary. Therefore, we obtain the following corollary.
\begin{cor}
Let $X$ be a smooth projective variety of dimension $n$ with $D$ a simple normal crossing divisor and $V$ a polarizable real VHS on $X\setminus D$. If $L$ is an ample divisor on $X$, then the locally free sheaf 
$$S^{\lee 0}(V)\otimes \omega_X(kL)$$ 
is globally generated when $k\ge n+1$ if $n\le5$, or $k\ge$$n+1\choose 2$ in general. 
\end{cor}
Geometrically, when $V$ is the $(d+i)$-th geometric VHS given by a family $f$ for $f: Y\to X$ being a surjective projective morphism branched along $D$, we know, as mentioned previously,
$$S^{\lee 0}(V)=R^if_*\omega_{Y/X}.$$
Hence, the above corollary in particular gives global generation for 
$$R^if_*\omega_Y\otimes \cO_X(kL),$$
which is obtained already in \cite{Kaw02} (besides the case of dimension 5) from a different point of view using the geometric Hodge metric and Kawamata coverings. 

Using Artin-Grothendieck vanishing for construtible sheaves on affine varieties and strictness of direct-images of mixed Hodge modules, we also deduce the following Nakano-type vanishing for $\R$-divisors.
\begin{thm}\label{thm:NakanoVan}
Let $X$ be a smooth projective variety and let $E=\sum E_i$ be a simple normal crossing divisor. Assume that $\Delta=\sum \alpha_i E_i$ is an $\R$-divisor supported on $E$ and $0<\alpha_i\le 1$ for each $i$ and $L$ is a divisor. Then we have
\begin{enumerate}
\item{if $L+\Delta$ is semi-ample (that is, $L+\Delta$ is $\R$-linear equivalent to a semi-ample divisor) and $X\setminus E_i$ is affine for some irreducible component $E_i$ of support of $\Delta$ with the coefficient $a_i\neq 1$, then 
\[H^p(X, \Omega^q(\log~D)(L))=0, \textup{ for } p+q>n;\]}\label{item:semi-ample-nakano}
\item{if $L+\Delta$ is ample, then 
\[H^p(X, \Omega^q(\log~D)(L))=0, \textup{ for } p+q>n.\]}\label{item:ample-nakano}
\end{enumerate}
\end{thm}
Similar Nakano-type vanishing results have been achieved by many authors from different points of views. For instance, the case in Item \eqref{item:semi-ample-nakano} when $\Delta=D$ has been studied in \cite{PM17} following elementary methods; the cases under more restrictive conditions in both Item \eqref{item:semi-ample-nakano} and Item \eqref{item:ample-nakano} can be found in \cite[\S 6]{EV} using positive characteristic methods of Deligne-Illusie; When $\Delta=D$, Item \eqref{item:ample-nakano} recovers Steenbrick vanishing \cite{St}; moreover, Item \eqref{item:ample-nakano} has been widely generalized to the compact K\"ahler case for $k$-positive bundles in \cite{HLWY} from a metric approach; see also \cite{AMPW} in the $\Q$-divisor case elementarily using Steenbrink vanishing and Kawamata coverings and \cite{Ara04} along the line of positive characteristic methods.   

\noindent
\subsection{Why real Hodge modules?}
Real Hodge modules are more natural to be considered on complex manifolds, for instance Hodge modules given by variations of Hodge structures parametrizing the cohomology of a compact K\"ahler family over a complex manifold. We work with $\R$-Hodge modules (rather than $\Q$-Hodge modules) and their injectivity and vanishing over projective spaces for the following reasons. First, the construction of Deligne lattices and Deligen meromorphic extensions of flat bundles is analytic over complex manifolds; see \S2.1. The second reason is about obtaining local freeness of the (filtered) Deligne lattice of a polarizable VHS, which is required for the construction of Hodge modules from a generically defined VHS (see Theorem \ref{thm:SaiStr}). To be more precise, the local freeness depends only on the existence of the polarization on the VHS, but not on the defining field. In fact, the Hodge metric given by the polarization makes the Deligne lattice of the VHS locally free; see Lemma \ref{lem:main}. 
In the case that polarizations exist on the VHS (no matter which field it is defined over), the eigenvalues of the (local) monodromies of the underlying local system are forced to be  of absolute value 1, which makes the index of the Deligne lattices $\R$-divisors and hence the induced $\Dmod$-modules and Hodge modules (if exist) $\R$-specializable in the sense of Sabbah \cite{SabVan}; see \S 2.3 for details. Hence, it is natural to consider $\R$-Hodge modules induced from polarizable VHS instead of the rational ones even over algebraic manifolds.  Moreover, being in the real case for $\R$-Hodge modules makes the injectivity and vanishing more general.

\subsection{Structure}
We summarize the structure of this article. \S 2 is the main section of the article. We discuss Deligne lattices and intermediate mixed Hodge module extensions of  polarizable $\R$-VHS and their counterparts in the twisted case. In \S3, we give proofs of the main theorems and \S4 is about applications of our vanishing and injectivity.

\noindent 
{\bf Acknowledgement.} 
The author is deeply grateful to his former Ph.D advisor Mihnea Popa for his continuous encouragement and for many useful discussions. The author also thanks Christopher Hacon, Linquan Ma, Karl Schwede and Xiaokui Yang for discussions and for answering questions. 
\section{$\R$-Hodge modules in normal crossing case}

\subsection{Canonical construction of Deligne lattices and Deligne meromorphic extensions}
The construction of Hodge modules highly depends on the well-known theory of Deligne extensions of local systems and more generally of variations of Hodge structures. For completeness and late use, we give the construction of Deligne lattices and Deligne meromorphic extensions using $L^2$ technics inspired by ideas in \cite[Ch. VI]{Bj}. 

 
Suppose that $X$ is a complex manifold of dimension $n$ with a normal crossing divisor $D=\sum D_i$ (that is locally $D=(z_1\cdots z_r=0)$ with coordinates $(z_1,\dots,z_n)$) throughout this section. Let $\cV$ be a holomorphic vector bundle with a flat connection $\nabla$ on $X\setminus D$. 

Since ultimately we are interested in $\R$-Hodge module, we are always assuming that the eigenvalues of the local monodromies of the flat lattice of $\cV$ (the underlying local system) are complex numbers of absolute value 1. However, most of the results about Deligne lattices in this article still hold in the general situation with very minor modification of the arguments.  

\begin{de}[Good Coverings]
An open covering $\{U_\alpha\}$ of $X\setminus D$ is said to be good if the following conditions are satisfied.
\begin{enumerate}
\item{Every $U_\alpha$ is simply connected}
\item{For every relatively compact subset $K$ of $X$, there are only finitely many $U_\alpha$ having non-empty intersections with $K$.}
\end{enumerate}
\end{de}

\begin{rmk} \label{goodcover}
Good coverings for $X\setminus D$ always exist using resolution of singularities, even when $D$ is not normal crossing. See \cite[Lemma 4.1.2]{Bj}.
\end{rmk}
Fix a good covering $\{U_\alpha\}$ of $X\setminus D$. Suppose  $\{c_j^\alpha\}_{j=1}^m$ is a set of horizontal sections trivializing $\cV|_{U_\alpha}$ for each $\alpha$ (since $U_\alpha$ is simply connected). For a section $s\in \Gamma(U, j_*\cV)$ over an open subset $U$ of $X$, we know
\begin{equation}\label{secmod1}
s|_{U\cap U_\alpha}=\sum\limits_j f_j^\alpha c_j^\alpha,
\end{equation}
where $j: X\setminus D\hookrightarrow X$ is the open embedding and $f_j^\alpha\in \Gamma(U\cap U_\alpha, \cO_X).$
\begin{de}\label{def:l2trivial}
For any $\R$-divisor $B$ supported on $D$, a section $s\in  \Gamma(U, j_*\cV)$ is said to be locally $L^2$ along $ D$ with weight $B$ if for a good covering $\{U_\alpha\}$ of $U$ and every $z_0\in U\cap D$, there exists a polydisk neighborhood $\Delta^n$ of $z_0$ with coordinates $(z_1,\dots,z_n)$ such that $B|_{\Delta^n}=\sum t_iD_i|_{\Delta^n}$ and $D_i|_{\Delta^n}=(z_i=0)$, and
\[\int_{\Delta^n\cap U_\alpha}\dfrac{\abs{f_j^\alpha}^2}{\Pi_i \abs{z_i}^{2t_i}}du< \infty\]
for every pair of $\alpha$ and $j$. Here $du$ is the volume form on $X$ and $f^\alpha_j$ are as in \eqref{secmod1}. 
\end{de}
One can easily check that the above $L^2$ condition is independent of choices of good covers and horizontal trivializaitions. Denote by $L^2(\cV,D;B)$ the subsheaf of $j_*\cV$ consisting of locally $L^2$ sections with the weight $B$. By Riemann integral formula, one can check the flat connection on $\cV$ extends to a flat logarithmic connection on $L^2(\cV,D;B)$ with poles along $D$. Here we conventionally set $B=B+0\cdot D$; that is, the coefficient of $D_i$ for every $i$ needs to be considered even if it is $0$.

For instance, when $\cV=\cO_{X\setminus D}$ and $B$ is effective, the $L^2$ construction yields multiplier ideals; that is, $$L^2(\cV,D;B)=\cI(\psi_B)=\cO_X(-\floor{B}),$$ where $\psi_B=\sum t_i\log\abs{g_i}$ is the plurisubharmonic function associated to $B$ locally, and $\cI(\psi_B)$ is the analytic multiplier ideal associated to $\psi_B$. 

When we want to birationally modify the base space, since $\cV$ is defined generically, we will still use $\cV$ to denote its pullback over the new base by abuse of notations. Since $L^2(L,D;B)$ is independent of choices of good covers, we immediately obtain the following lemma, which will be needed later.
\begin{lemma}\label{lem:birDLattice}
If $\mu: X'\rightarrow X$ is a log resolution of $(X, D)$ with $E=(\mu^*D)_{\textup{red}}$, then we have for every $\R$-divisor $B$ supported on $D$
\[\mu_*(L^2(\cV,E;\mu^*B)\otimes \omega_{X'})=L^2(\cV,D;B)\otimes \omega_X.\]
\end{lemma}

\begin{prop}
For every $\R$-divisor $B$ supported on $D$, $L^2(\cV,D;B)$ is a locally free $\cO_X$-module of finite type. 
\end{prop}

\begin{proof}
Assume $D=(z_1\cdots z_r=0)$ and $B=\sum_{i=1}^r t_iD_i$ with coordinates $(z_1,\dots,z_n)$ on a polydisk neighborhood $\Delta^n$. Choose a multivalued horizontal frame $\{c_j\}$ of $\cV$ on $\Delta^n\setminus D$ with monodromy $\gamma_i$ along each $z_i$ for $i=1,\dots, r$, which trivialize $\cV$ on each simply connected neighborhood of $\Delta^n\setminus D$. 
Then we set $$s^{t}_j=e^{\sum_{i=1}^r\Gamma^{t_i}_i\text{log}~z_i}\cdot c_j.$$
Here $\Gamma^{t_i}_i$ is the logarithm of $\gamma_i^{-1}$ (namely $\gamma_i=e^{-2\pi \sqrt{-1}\Gamma^{t_i}_i}$) with eigenvalues satisfying
\[t_i-1< \textup{eigenvalues of } \Gamma^{t_i}_i\le t_i.\]

By construction, one can easily see that all $s^{t}_j$ are univalued and $L^2$ with weight $B$, that is, 
\[s^{t}_j \in \Gamma(\Delta^n, L^2(\cV,D;B)).\]
On the other hand, assume $s\in \Gamma(X, L^2(\cV,D;B))$. Take a good covering $\{U_C\}_C$. On each $U_C$, $s$ can be written as
\[s|_{U_C}=\sum_jf^C_jc_j,\]
where $f^C_j$ are holomorphic functions on $U_C$. But 
\[s|_{U_C}=\sum_jf^C_jc_j=\sum_jf^C_je^{-\sum_{i=1}^r\Gamma^{t_i}_i\text{log}~z_i}\cdot s^t_j|_{U_C}=\sum_jg^C_js_j^t.\]
Since $s_j^t$ are globally defined on $\Delta^n\setminus D$ and linear independent, we see all $g^C_j$ glue together as a holomorphic function $g_j$ on $X\setminus D$.
Since
\[\int_{\Delta^n\cap U_C}\dfrac{\abs{f_j^C}^2}{\Pi \abs{z_i}^{2t_i}}du< \infty,\]
by Cauchy-Schwarz inequality and the eigenvalue assumption on $\Gamma^{t_i}_i$ we obtain
\[\int_{\Delta^n} \abs{g_j}^2du <\infty, \]
which means $g_j$ extends to a holomorphic function on $\Delta^n$. Therefore, the sections in $\{s^t_j\}$ trivialize $L^2(\cV,D;B)$.
\end{proof}
Since $L^2(\cV,D;B)$ is locally free, we define, for every $\R$-divisor $B=\sum t_i D_i$ supported on $D$  the (upper and lower) $Deligne$ $lattices$ by
\begin{equation}\label{eq:upperext}
\cV^{\lee B}=L^2(\cV,D;B)
\end{equation}
and
\begin{equation}\label{eq:lowerext}
\cV^{\gee B}=L^2(\cV,D;B+(1-\epsilon)D)
\end{equation}
for $0<\epsilon\ll1$. It is clear that $L^2(L,D;B+(1-\epsilon)D)$ is independent of $0<\epsilon\ll1$ from the above proof. Moreover, it is worth noticing that $\Gamma^{(1-\epsilon)t_i}_i=\Res_{D_i}\nabla$. Hence, the $\cV^{\lee B}$ (resp. $\cV^{\gee B}$) has the characterization of eigenvalues, 
\[t_i-1<(\textup{resp. } t_i\le)\textup{the eigenvalues of $\Res_{D_i}\nabla$}\le t_i(\textup{resp. }<t_i+1) .\]

Immediately from definition, we have the following corollary.
\bc\label{cor:semiconfil}
If two $\R$-divisors $B_1\leq B_2$ supported on $D$, then we have for every $i$
\be
\item{$\cV^{\lee B_1+D_i}=\cV^{\lee B_1}(-D_i)$ and $\cV^{\gee B_1+D_i}=\cV^{\gee B_1}(-D_i)$;}
\item{$\cV^{\lee B_1+\epsilon\cdot D_i}=\cV^{\lee B_1}$ and $\cV^{\gee B_1-\epsilon\cdot D_i}=\cV^{\gee B_1}$ for $0<\epsilon\ll1$;}
\item{$\cV^{\lee B_1+D}\subseteq$ $\cV^{\gee B_1}$;}
\item{$\cV^{\lee B_2}\subseteq\cV^{\lee B_1}$ and $\cV^{\gee B_2}\subseteq \cV^{\gee B_1}$;}
\item{$(\cV^{\lee B_1})^*\simeq$ ${(\cV^*)}^{\gee -B_1},$}
where $\bullet^*$ denotes the duality functor of $\cO$-modules.
\ee
\ec
Then the $Deligne$ $meromorphic$ $extension$ of $\cV$ is given by
\begin{equation}\label{eq:meroext}
\cV(*D)=\varinjlim_B \cV^{\lee B}=\varinjlim_B\cV^{\gee B}=\cV^{\lee B}\otimes_\cO\cO_X(*D)=\cV^{\gee B}\otimes \cO_X(*D)
\end{equation}
for arbitrary $\R$-divisor $B$ supported on $D$, where $\cO_X(*D)$ is the sheaf of meromorphic functions that are holomorphic away from $D$. The logarithmic connection induces a flat connection on the meromorphic extension which makes it a $\Dmod_X$-module. In fact, it is regular holonomic. 

In summary, the above corollary says that $\{\cV^{\lee B}\}_B$ (resp. $\{\cV^{\gee B}\}_B$) is an exhausted multi-indexed upper (resp. lower) semi-continuous and locally discrete discreasing filtration on $\cV(*D)$.

\subsection{Deligne lattice and $\R$-indexed $V$-filtrations on meromorphic extensions}
In this paper, we work with real Hodge modules, whose underlying regular holonomic $\Dmod$-modules are $\R$-specializable along any hypersurfaces in the sense of Sabbah. $\R$-$specializablity$ means that the eigenvalues of the action of the Euler vector field along the hypersurface on the associated graded module of the Kashiwara-Malgrange filtration are all real numbers; see \cite{SabVan}; see also \cite{Wu17} for a detailed discussion. In particular, the Kashiwara-Malgrange filtration has an $\R$-indexed refinement, called the $\R$-indexed $V$-filtration. Let us first recall its definition.  

We consider a $\bbC$-subalgebra of $\Dmod_X$ for a normal crossing divisor $D$, given by 
\[V^0_D\Dmod_X=\{P\in \Dmod_X|P\cdot I^j_D\subseteq I^j_D\textup{ for every $j\in \Z$}\}\]
where $I_D$ is the ideal sheaf of $D$ with the convention that $I^j_D=\cO_X$ for $j\le 0$. The order filtration $F_\bullet$ of $\Dmod_X$ induces a filtration $F_\bullet$ on $V^0_D\Dmod_X$.

\begin{de}
The $\R$-indexed  $V$-filtration along a smooth hypersurface $H$ on a left $\Dmod_X$-module $M$ is a $\R$-indexed (locally discrete, that is the filtration jumps discretely locally on a relative compact neighborhood) decreasing filtration $\d{V}M$ such that
\be
\item $\bigcup V^\alpha M=M$ and $V^\alpha M$ is coherent over $V_H^0\Dmod_X$;
\item $t\cdot V^\alpha M\subset  V^{\alpha+1}M$, $\partial_t\cdot V^\alpha M\subset V^{\alpha-1}M$, for every $\alpha\in \R$;
\item $t\cdot V^\alpha M = V^{\alpha+1}M$ for $k\gg 0$; 
\item $t\partial _t$ (the Euler vector field along $H$) acts on $\gr_{V}^\alpha{M}=\frac{V^\alpha M}{V^{>\alpha }M}$ locally nilpotently (globally in the algebraic case), where $V^{>\alpha }M=\cup_{\beta>\alpha}V^{\beta}M$.
\ee 
\end{de}

When considering the meromorphic extention $\cV(*D)$, by the well-known Riemann-Hilbert correspondence of nearby and vanishing cycles (see for instance \cite[\S4.6]{Wu17}), our assumption on local monodromies implies that $\cV(*D)$ is $\R$-specilizable along every smooth hypersurface. Indeed, in this situation the Deligne lattices contain exactly the same information as the $\R$-indexed V-filtrations do. 
\begin{prop}\label{prop:vfDe}
Let $X$ be a complex manifold, and $D=\sum_i D_i$ a simple normal crossing divisor (i.e. $D$ is normal crossing and each irreducible component $D_i$ is smooth). Suppose that $\cV$ is a flat holomorphic vector bundle $X\setminus D$ so that eigenvalues of the local monodromies of the flat lattice are of absolute value 1. Then the Deligne meromorphic extension $\cV(*D)$ is $\R$-specializable along each $D_i$. Moreover, we have for every $\R$-divisor $B=\sum \alpha_iD_i$,
\[\cV^{\gee B}=\bigcap_{i=1}^rV^{\alpha_i}_{D_i}\cV(*D),\]
where $V^{\bullet}_{D_i}\cV(*D)$ is the V-filtration along $D_i$.
\end{prop}

\bp
For every $\alpha\in \R$, we set
$$V^{\alpha}_{D_i}\cV(*D)=\cV^{\gee\alpha D_i}\otimes_\cO \cO_X(*(D-D_i)).$$ 
One verifies easily, $\{V^{\alpha}_{D_i}\cV(*D)\}_\alpha$ defines the $V$-filtration on $\cV(*D)$ along $D_i$. The second statement is then obvious. 
\ep

\subsection{$\R$-Hodge modules in the normal crossing case}
The theory of Hodge modules was originally developed with $\Q$-coefficients in \cite{Sai88}. To study K\"ahler family, Saito \cite{Sai90a} also generalized the theory to the case of $\R$-coefficients. The underlying regular holonomic $\Dmod$-modules are $\R$-specializable in this situation. See also \cite[Ch. B]{PPS}.

We first recall some terminology. Assume as in previous section that $X$ is a complex manifold of dimension $n$ with a normal crossing divisor $D$. 

The basic input of a real Hodge module $M$ is a triple $(\cM, F_\bullet, K_\R)$ satisfying 
\begin{enumerate}
\item $\cM$ is a regular holonomic left $\Dmod_X$-module 
\item $F_\bullet$ is a coherent filtration on $\cM$
\item $K_\R$ is and $\R$-perverse sheaf satisfying $$\DR(\cM)\simeq K_\R\otimes \bbC$$
in the sense of Riemann-Hilbert correspondence.
\end{enumerate} 
The triple is an $\R$-$Hodge$ $module$ if moreover an inductive requirement for nearby cylces along any holomorphic functions is satisfied; see \cite{Sai88} and \cite{Sai90a}. A $polarization$ on a weight $w$ $\R$-Hodge module $M$ is an isomorphism between its dual and  its $w$-th Tate twist subject to a similar inductive requirement; one says that $M$ is polarizable if it admits at least one polarization.  For instance, $\R^H=(\cO_X, F_\bullet, \R[n])$ is the trivial Hodge module, with the trivial filtration $F_p\cO_X=\cO_X$ for $p\ge 0$ and $0$ otherwise.

An $\R$-Hodge module is $strictly$ supported on $X$ if it has no nontrivial subobjects or quotient objects that are supported on a proper subset. One of the most important result of Hodge modules is the structure theorem relating polarizable Hodge modules and polarizable variations of Hodge structure (VHS). We recall it for the latter use.
\begin{thm}[Saito]\label{thm:SaiStr}
The category of polarizable real Hodge modules of weight $w$ strictly supported on a singular variety $X$ is equivalent to the category of generically defined polarizable real variations of Hodge structure of weight $w-\dim X$ on $X$.
\end{thm}
It is originally proved in \cite{Sai90b} in rational case. Since the proof in $loc.$ $cit.$ rests on the nilpotent orbit theorem and $SL_2$-orbit theorem of \cite{Sch73} and the multivariable generalization \cite{CKS}, which can be generalized to real-coefficient case as noticed in \cite[App.]{SV11}(see also \cite{PPS}), the proof carries over to the real-coefficient case; see  \cite{Sai90a} for details.
 
For the application of this article, we discuss the construction of polarizable Hodge modules from polarizable variations of Hodge structure in the normal crossing case. 

Let $V=(\cV, F_\bullet, \VV_\R)$ be a polarizable real VHS on $X\setminus D$. After fixing a polarization, the eigenvalues of the local monodromies along components of $D$ are forced to be of absolute value 1 (see for instance \cite[Proof of Lem. 4.5]{Sch73}; it holds even for polarizable complex VHS). Namely, the assumption on eigenvalues is automatic if the flat bundle underlies a polarizable real VHS. 

Naively, for an $\R$-divisor $B$ supported on $D$, we define 
$$F^{\lee B}_\bullet=\cV^{\lee B}\cap j_*F_\bullet,$$
and 
$$F^{\gee B}_\bullet=\cV^{\gee B}\cap j_*F_\bullet,$$
where $j\colon X\setminus D\to X$ the open embedding, and set $V^{\lee B}=(\cV^{\lee B},F_\bullet^{\lee B}, V_\R)$ and $V^{\gee B}=(\cV^{\gee B},F_\bullet^{\gee B}, V_\R)$, the filtered Deligne lattices of the VHS. 

By construction, Lemma \ref{lem:birDLattice} implies immediately the following birational invariant property of $F^{\lee B}_\bullet$. 
\begin{lemma}\label{lem:birSV}
If $\mu: X'\rightarrow X$ is a log resolution of $(X, D)$ with $E=(\mu^*D)_{\textup{red}}$, then we have for every $\R$-divisor $B$ supported on $D$
\[\mu_*(F^{\lee \mu^*B}_\bullet\otimes \omega_{X'})=F^{\lee B}_\bullet\otimes \omega_X.\]
In particular, we have 
\[\mu_*(S^{\lee \mu^*B}(V)\otimes \omega_{X'})=S^{\lee B}(V)\otimes \omega_X.\]
\end{lemma}

It is not clear a priori that $F^{\lee B}_\bullet$ or $F^{\gee B}_\bullet$ are coherent (indeed, their coherence implies local freeness; see \cite[\S 3.b]{Sai90b}). We justify their local freeness because this point seems to be missing in literature. 
\begin{lemma} \label{lem:main}
With notations as above, for every $\R$-divisor $B$ supported on $D$, $\{F^{\lee B}_\bullet\}$ (resp. $\{F^{\gee B}_\bullet\}$) is a subbundle filtration of $\cV^{\lee B}$ (resp. $\cV^{\gee B}$).
\end{lemma}
\begin{proof}
When the monodromies are quasi-unipotent, Saito's idea is to reduce it to the unipotent case, using covering tricks and compatibility of the $F$-filtration and $V$-filtrations along components of $D$. In the unipotent case, the above lemma follows from the nilpotent orbit theorem of Schmid. In our case, covering tricks do not apply anymore; instead, we can appeal to the theory of tame harmonic bundles in the sense of Simpson and Mochizuki. 

From the classical theory of variations of Hodge structure, we know that $V$ gives rise to a harmonic bundle $(\cE, \theta, h)$ in the sense of Simpson \cite{Sim90}, where $\cE$ is the associated graded module, $\theta$ the Higgs field induced by the flat connection, and $h$ the Hodge metric induced by the polarization. 

Let us recall some terminology. Similar to the definition of locally $L^2$ sections with respect to the flat metric (cf. Definition \ref{def:l2trivial}), for an $\R$-divisor $B=\sum_i \alpha_iD_i$ supported on $D$, we say a section $s$ of $j_*\cE$ (resp. $j_*\cV$) has the order of growth bigger than $B$ if there exists a positive constant $C$ so that
\[\norm s_h\le C\cdot\prod_i \abs z_i^{\alpha_i-\epsilon}\] 
locally when $D_i$ is defined by a coordinate $z_i$, for $0<\epsilon\ll 1$.  Denote by $\cE^B$ (resp. $\cV^B_h$) the subsheaf of $j_*\cE$ (resp. $j_*\cV$) consisting of sections with order of growth bigger than $B$. Clearly, the $\{\cE^B\}_B$ (resp. $\{\cV^B_h\}_B$) form a decreasing filtration on 
\[\cE(*D)=\bigcup_B \cE^B \textup{ (resp. }\cV(*D)_h= \bigcup_B \cV^B_h).\] 

Since $\theta=\theta_\bullet$ is a graded morphism, it is nilpotent. In particular, the harmonic bundle $(\cE, \theta, h)$ is tame. We refer to \cite{Sim90} for the definition of tameness for curves and \cite{Moc06} in general. By \cite[Prop. 2.53]{Moc09}, we have that $\cE(*D)$ together with $\{\cE^B\}_B$ is a parabolic bundle and hence so is $\cV(*D)_h$ with $\{\cV_h^B\}_B$; see \cite{Moc06} for the definition of parabolic bundles. Thanks to the nilpotency of $\theta$ again, the meromorphic sheaves $\cV(*D)_h$ and $\cV(*D)$ as well as the filtrations $\{\cV^B_h\}_B$ and $\{\cV^{\gee B}\}_B$ are canonical isomrophic; see the table in \cite[Page 720]{Sim90} or more precisely \cite[Lem. 2.12]{Bru17}. In particular, since $\cE^B$ and $\cV^B_h$ are (non-canonically) isomorphic locally, the grading structure on $\cE^B$ induces a subbundle filtration $\{F_\bullet^{\lee B}\}$ on $\cV^{\gee B}$. Necessarily we have $F_\bullet^{\gee B}=\cV^{\gee B}\cap j_*F_\bullet$.  Since locally
$$F_\bullet^{\lee B}=F_\bullet^{\gee B-(1-\epsilon)D}$$
for $0<\epsilon \ll 1$, we conclude that $F_\bullet^{\lee B}$ is also a subbundle filtration. 
   
\end{proof}

As a byproduct, the proof above implies that we have a canonical isomorphism
\[\cE^B\simeq \gr^F_\bullet \cV^{\gee B}.\]

Meanwhile, we have a $\Dmod_X$-module defined by $\Dmod_X\cdot \cV^{\lee 0}$
with the filtration by convolution, that is,
\[F_p (\Dmod_X\cdot \cV^{\lee 0})=\sum_i F_{p-i}\Dmod_X\cdot F_i^{\lee 0}.\]
Using Riemann-Hilbert correspondence, one checks that the underlying perverse sheaf is the minimal perverse extension $j_{!*}\VV_\R[n]$, where $n$ is the dimension of $X$. Now, the proof of \cite[Thm. 3.21]{Sai90b} carries over to the real case and we see that the filtered $\Dmod_X$-module $(\Dmod_X\cdot \cV^{\lee 0}, F_\bullet)$, underlies the Hodge module uniquely extending the VHS as in Theorem \ref{thm:SaiStr}. We denote it by $j_{!*}V$ and refer it as the minimal extension of $V$.

\subsection{Deligne lattice of mixed type and extension of scalars}
We assume that $\cV$ is a flat vector bundle on $X\setminus D$, where $X$ is a complex manifold and $D$ a normal crossing divisor as before. First, we introduce Deligne lattices of mixed-type. 

By the semicontinuity of Deligne lattices $\cV^{\lee B}$ and $\cV^{\gee B}$, we are allowed to perturb (positively or negatively) the coefficients of the index divisor $B$. To be precise, we define Deligne lattices of mixed-type as follows. If $J\subseteq I$ a subset of the index set of $D$, then for a pair of $\R$-divisors $(B_1, B_2)$ satisfying 
\begin{equation}\label{eq:pair1}
B_1=\sum_{i\in J} t_i D_i
\end{equation}
and 
\begin{equation}\label{eq:pair2}
B_2=\sum_{i\notin J}t_i D_i,
\end{equation}
we denote by $\cV^{B_1^{\gee}+B_2^{\lee}}$ the Deligne lattice satisfying 
\[t_i\le\textup{the eigenvalues of $\Res_{D_i}\nabla$}<t_i+1 \textup{ if } i\in J,\]
and 
\[t_i-1<\textup{the eigenvalues of $\Res_{D_i}\nabla$}\le t_i \textup{ if } i\notin J.\]
For instance, if $J=I$ (resp. $J$ is empty), we recover the lower (resp. upper) Deligne lattice of index $B_1$ (resp. $B_2$) .  
  
Similar to the upper or lower Deligne lattices, one sees that the flat logarithmic connections on Deligne lattices $\cV^{B_1^{\gee}+B_2^{\lee}}$ induce $V_D^0\Dmod_X$-module structures extending the nature $\cO_X$-module structures.
\begin{prop}\label{prop:shiftind}
With notations as above, for a fixed $J\subseteq I$, we have the following
\begin{enumerate}
\item{if $0$ is not an eigenvalue of the residue map along $D_i$ for some $i\in J$, then 
\[\Dmod_X\otimes_{V_D^0\Dmod_X}\cV^{B_1^{\gee}+B_2^{\lee}}=\Dmod_X\otimes_{V_D^0\Dmod_X}\cV^{(B_1+ a\cdot D_i)^{\gee}+B_2^{\lee}}\]
for $a\in \R$.}
\item{if $0$ is not an eigenvalue of the residue map along $D_i$ for some $i\notin J$, then 
\[\Dmod_X\otimes_{V_D^0\Dmod_X}\cV^{B_1^{\gee}+B_2^{\lee}}=\Dmod_X\otimes_{V_D^0\Dmod_X}\cV^{B_1^{\gee}+(B_2+a\cdot D_i)^{\lee}}\]
for $a\in \R$.}
\item{ If $B_1+\sum_{i\in J}D_i\le0$ and $B_2<0$, then 
\[\Dmod_X\otimes_{V_D^0\Dmod_X}\cV^{B_1^{\gee}+B_2^{\lee}}=\cV(*D)\]}
\item{ If $B_1>0$, then if $i\in J$, we have
\[\Dmod_X\otimes_{V_D^0\Dmod_X}\cV^{B_1^{\gee}+B_2^{\lee}}=\Dmod_X\otimes_{V_D^0\Dmod_X}\cV^{(B_1+a\cdot D_i)^{\gee}+B_2^{\lee}}\]
for $a\in \R^+$.}
\item{ If $B_1+\sum_{i\in J}D_i\le0$, then if $i\in J$, we have
\[\Dmod_X\otimes_{V_D^0\Dmod_X}\cV^{B_1^{\gee}+B_2^{\lee}}=\Dmod_X\otimes_{V_D^0\Dmod_X}\cV^{(B_1+a\cdot D_i)^{\gee}+B_2^{\lee}}\]
for $a\in \R^-$.}
\item{ If $B_2\ge \sum_{i\notin J}D_i$, then if $i\in I\setminus J$, we have
\[\Dmod_X\otimes_{V_D^0\Dmod_X}\cV^{B_1^{\gee}+B_2^{\lee}}=\Dmod_X\otimes_{V_D^0\Dmod_X}\cV^{B_1^{\gee}+(B_2+a\cdot D_i)^{\lee}}\]
for $a\in \R^+$.}
\item{ If $B_2<0$, then if $i\in I\setminus J$, we have
\[\Dmod_X\otimes_{V_D^0\Dmod_X}\cV^{B_1^{\gee}+B_2^{\lee}}=\Dmod_X\otimes_{V_D^0\Dmod_X}\cV^{B_1^{\gee}+(B_2+a\cdot D_i)^{\lee}}\]
for $a\in \R^-$.}
\end{enumerate}
\end{prop}
\begin{proof}
All the statements follows by induction on the rank of the underlying local system and in rank 1 case follows from local computation. The details are left to interested readers. See also \cite[\S IV 2]{Bj} and \cite[\S 3 and 4]{WuThesis} 
\end{proof}

\subsection{Intermediate extensions and Mixed Hodge modules}
In this section, we assume $V=(\cV, F_\bullet, \VV_\R)$ is a polarizable $\R$-VHS on $X\setminus D$, where $X$ is a complex manifold and $D$ a normal crossing divisor as before. 

By Lemma \ref{lem:main} and the construction, the filtration 
$$\{F_\bullet^{B_1^{\gee}+B_2^{\lee}}=j_*F_\bullet\cap \cV^{B_1^{\gee}+B_2^{\lee}}\}$$ is a subbundle filtration on $\cV^{B_1^{\gee}+B_2^{\lee}}$.

Naturally, the $\Dmod_X$-module $\Dmod_X\otimes_{V_D^0\Dmod_X}\cV^{B_1^{\gee}+B_2^{\lee}}$ has a filtration given by convolution
\begin{equation}\label{eq:cf}
F_p(\Dmod_X\otimes_{V_D^0\Dmod_X}\cV^{B_1^{\gee}+B_2^{\lee}})=\sum_i F_{p-i}\Dmod_X\otimes F_i^{B_1^{\gee}+B_2^{\lee}}.
\end{equation}

Saito \cite{Sai90b} constructed mixed Hodge module in rational case originally. But the construction does not depend on the defining fields. Therefore, $\R$-mixed Hodge module can be defined exactly in the same way; see \S2.d in $loc.$ $cit.$ In particular we have functors of extensions over analytic subspaces. To be precise, if $j_Z: X\setminus Z\to X$ the open embedding for a codimension $1$ analytic subspace $Z\subset X$ and $M$ a $\R$-mixed Hodge module on $X$, then $j_{Z*}j_Z^{-1}(M)$ and $j_{Z!}j_Z^{-1}(M)$ are all $\R$-mixed Hodge module and the functors $j_{Z*}j_Z^{-1}(\bullet)$ and $j_{Z!}j_Z^{-1}(\bullet)$ are compatible with the corresponding functors on perverse sheaves.

In our case, we would discuss various mixed Hodge modules constructed from $V$, the polarizable VHS, using the functors $j_{Z*}\circ j_Z^{-1}$ and $j_{Z!}\circ j_Z^{-1}$. 

For $J\subseteq I$, we set $D_J=\sum_{i\in J}D_i$, $j_{1_J}\colon X\setminus D\to X\setminus D_J$ and $j_{2_J}\colon X\setminus D_J\to X$. Then by Theorem \ref{thm:SaiStr}, for every $J\subseteq I$, we get the mixed Hodge module ${j_{1_J}}_*{j_{2_J}}_!V$.
\begin{prop}\label{prop:intermex}
With above notations, for every $J\subseteq I$, we have that
\[(\Dmod_X\otimes_{V_D^0\Dmod_X}\cV^{(-D_J)^{\gee}+(0\cdot D_{I\setminus J})^{\lee}}, F_\bullet, {Rj_{1_J}}_*{j_{2_J}}_!\VV_\R[n])\]
is the mixed Hodge module ${j_{1_J}}_*{j_{2_J}}_!V$, where the filtration is given by convolution as in \eqref{eq:cf}. 
\end{prop} 
\begin{proof}
The mixed Hodge module ${j_{1_J}}_*{j_{2_J}}_!V$ can be glued from nearby cycles of $V$ along each $D_i$ using Beilinson functors. One observes that all the nearby cycles are of the same type as ${j_{1_J}}_*{j_{2_J}}_!V$ (they are all of normal crossing type) using the natural stratification of $(X,D)$. Therefore, the conclusion follows by induction and the combinatorial data of the mixed Hodge module and its nearby cycles with respect to the stratification. Let us refer to \cite[\S2.e and \S3.a]{Sai90b} for details. 
\end{proof}
\subsection{Logarithmic de Rham complexes and filtered comparisions}
In this section, we present comparisons between Deligne lattices and intermediate extensions in the sense of de Rham complexes. 

We first recall the definition of logarithmic $de$ $Rham$ complexes. Suppose $\cM$ is a $\cO_X$-module with a flat logarithmic connection $\nabla$ on $X$ with logarithmic poles along $D$, where $X$ is a complex manifold and $D$ a normal crossing divisor as before. The $logarithmic$ $de$ $Rham$ complex of $\cM$ is the complex starting from the $-n$-term
\[\DR_D(\cM)\colon=[\cM\to \Omega^1_X(\log D)\otimes \cM\to \cdots\to \Omega^n(\log D)\otimes \cM].\] 
If moreover $\cM$ is filtered with filtration $F_\bullet$ (assume that the filtration $F_\bullet\cM$ is compatible with the order filtration on $V^0_D\Dmod_X$), then $\DR_D(\cM)$ has the filtration
\[F_\bullet \DR_D(\cM)=[F_\bullet\cM\to \Omega^1_X(\log D)\otimes F_{\bullet+1}\cM\to \cdots\to \Omega^n(\log D)\otimes F_{\bullet+n}\cM.\]
If $D$ is empty, then it recovers the (filtered) $de$ $Rham$ complex for (filtered) $\Dmod_X$-modules.

\begin{prop}\label{prop:maincomp}
Assume $V=(\cV, F_\bullet, \VV_{\R})$ is a polarizable VHS on $X\setminus D$. For a subset $J\subset I$ with a pair $(B_1, B_2)$ as in \eqref{eq:pair1} and \eqref{eq:pair2}, we have canonical quasi-isomorphisms
\[\DR_D(\cV^{(B_1+D_J)^{\gee}+(B_2+D_{I\setminus J})^{\lee}})\to \DR(\Dmod_X\otimes_{V_D^0\Dmod_X}\cV^{B_1^{\gee}+B_2^{\lee}}),\]
and 
\[F_p \DR_D(\cV^{(B_1+D_J)^{\gee}+(B_2+D_{I\setminus J})^{\lee}})\to F_p\DR(\Dmod_X\otimes_{V_D^0\Dmod_X}\cV^{B_1^{\gee}+B_2^{\lee}})\]
for every $p$.
\end{prop}
\begin{proof}
It is enough to construct a graded quasi-isomorphism 
\[\gr_\bullet^F \DR_D(\cV^{(B_1+D_J)^{\gee}+(B_2+D_{I\setminus J})^{\lee}})\to\gr_\bullet^F\DR(\Dmod_X\otimes_{V_D^0\Dmod_X}\cV^{B_1^{\gee}+B_2^{\lee}}).\]
To this end, we first set 
$$\cE_{J,\bullet}=\gr_\bullet^F\cV^{(B_1+D_J)^{\gee}+(B_2+D_{I\setminus J})^{\lee}}, \widetilde\cE_{J,\bullet}=\gr_\bullet^F(\Dmod_X\otimes_{V_D^0\Dmod_X}\cV^{B_1^{\gee}+B_2^{\lee}}),$$
and 
$\cB_\bullet=\gr_\bullet^F\Dmod_X\otimes_{\gr_\bullet^F\cO_X}\cE_{J,\bullet}$ for simplicity.
Consider the graded complex 
\[\cC^\bullet_\bullet\colon\cB_{\bullet-n}\otimes\bigwedge^n\cT_X(-\log D)\to \cB_{\bullet-n+1}\otimes\bigwedge^{d-1}\cT_X(-\log D)\to\cdots\to \cB_\bullet,\]
which is locally the Koszul complex of $\cB_\bullet$ together with a sequence of actions 
$$x_1\partial_{x_1}\otimes 1-1\otimes x_1\partial_{x_1}, \dots,x_r\partial_{x_r}\otimes 1-1\otimes x_r\partial_{x_r}, \partial_{x_{r+1}}\otimes 1-1\otimes\partial_{x_{r+1}},\dots ,$$ 
where $\cT_X(-\log D)$ is the sheaf of holomorphic vector fields with logarithmic zeros along $D$, locally generated freely by $x_1\partial_{x_1}, \dots,x_r\partial_{x_r}, \partial_{x_{r+1}},\dots \partial_{x_n}.$ 

By Lemma \ref{lem:main}, we know $\cE_{J,\bullet}$ is locally free over $\gr_\bullet^F\cO_X$. Using local freeness of $\cE_{J,\bullet}$, one can check that the natural graded morphism 
\[\cC^\bullet_\bullet\to \widetilde\cE_{J,\bullet}\]
is quasi-isomorphic. On the other hand, since $\gr_\bullet^F\DR(\Dmod_X)$ is a resolution of $\omega_X\otimes_\cO\gr_\bullet^F\cO_X$, we know that
\[\gr_\bullet^F\DR(\Dmod_X)\overset{\bf L}\otimes_{\gr_\bullet^F\Dmod_X}\cC_\bullet^\bullet=\gr_\bullet^F\DR(\Dmod_X)\otimes_{\gr_\bullet^F\Dmod_X}\cC_\bullet^\bullet\]
is graded quasi-isomorphic to $\gr_\bullet^F \DR_D(\cV^{(B_1+D_J)^{\gee}+(B_2+D_{I\setminus J})^{\lee}})$. Therefore, we know $\gr_\bullet^F \DR_D(\cV^{(B_1+D_J)^{\gee}+(B_2+D_{I\setminus J})^{\lee}})$ is quasi-isomorphic to $$\gr_\bullet^F\DR(\Dmod_X)\otimes_{\gr_\bullet^F\Dmod_X}\widetilde\cE_{J,\bullet}.$$
Now, we finish the proof by observing the identification 
$$\gr_\bullet^F\DR(\Dmod_X)\otimes_{\gr_\bullet^F\Dmod_X}\widetilde\cE_{J,\bullet}\simeq \gr_\bullet^F\DR(\Dmod_X\otimes_{V_D^0\Dmod_X}\cV^{B_1^{\gee}+B_2^{\lee}}).$$
\end{proof}

\subsection{Twisted VHS and twisted Hodge modules}
We now discuss twisted VHS, their (filtered) Deligne lattices and twisted Hodge modules.  

Suppose that $V=(\cV, F_\bullet, \VV_\R)$ is a polarizable VHS on $U=X\setminus D$, where $X$ is a complex manifold with a normal crossing divisor $D$ as always. Let $\cL_U$ be a torsion line bundle on $U$ so that 
\[\cL_U^m\simeq\cO_U.\]
The $m$-th \'etale covering $\pi\colon U_{\cL_U}\to U$ of $\cL_U$ gives a  VHS (rational indeed)
\[V_{\cL_U}=(\bigoplus_{i=0}^{m-1}\cL_U^{-i}, F_\bullet, \pi_*\R_{U_{\cL_U}}),\]
with the filtration given by $F_p\cL_U^{-i}=\cL_U^{-i}$ for $p\ge 0$ and 0 otherwise. Then we have two new polarizable VHS given by tensor-products, 
$$V\otimes V_{\cL_U}=(\cV\otimes\bigoplus_{i=0}^{m-1}\cL_U^{-i}, F_\bullet, \VV_\R\otimes \pi_*\R_{U_{\cL_U}}),$$
and
$$V\otimes V^{*}_{\cL_U}=(\cV\otimes\bigoplus_{i=0}^{m-1}\cL_U^{i}, F_\bullet, \VV_\R\otimes (\pi_*\R_{U_{\cL_U}})^{*}),$$ 
where $V^{*}_{\cL_U}$ is the dual VHS and $(\pi_*\R_{U_{\cL_U}})^{*}$ is the dual local system, and the twisted VHS
\[V^i_{\cL_U}=(\cV\otimes\cL_U^{i},F_\bullet, \VV_\R\otimes L_U^i)\]
where $L_U^i$ is the rank-1 unitary representation associated to $\cL_U^{i}$ for $i=\pm1,\dots,\pm (m-1)$ and the filtrations are given by convolution. 

By considering direct summands of Deligne lattices and (mixed) Hodge modules associated to $V\otimes V_{\cL_U}$, it now makes sense to say twisted Deligne lattices and twisted Hodge modules (minimal or intermediate extensions) associated to $V^i_{\cL_U}$ with obvious definitions. 

One of the most important properties of mixed Hodge modules is that the Hodge filtration $F_\bullet$ is strict under direct image functors for projective morphisms; see \cite[Thm 2.14]{Sai90b}. Since the proof in $loc.$ $cit,$ does not depend on the defining field of the perverse sheaves, the strictness property still holds for real Mixed Hodge modules. When the base space is projective, strictness of the direct image with respect to the constant morphism is equivalent to the $E_1$-degeneration of the Hodge-to-de-Rham spectral sequence. In particular, we get the following useful lemma for the intermediate extensions of $V^i_{\cL_U}$ thanks to Proposition \ref{prop:intermex}.
\begin{lemma}\label{lem:H-DR}
Suppose that $X$ is projective manifold with a simple normal crossing divisor $D=\sum_{i\in I}D_i$. Let $V=(\cV, F_\bullet, \VV_\R)$ be a polarizable VHS on $U=X\setminus D$ and let $\cL_U$ be a torsion line bundle on $U$ so that $\cL_U^m\simeq\cO_U$.  For each $J\subseteq I$ and $i$, we have that the Hodge-to-de-Rham spectral sequence
\[E_1^{p,q}=\mathbb H^{p+q}(X,\gr^F_{-q}\DR(\cM_i))\Rightarrow \mathbb H^{p+q}(X,\DR(\cM_i))\]
degenerates at $E_1$, where 
$$(\cM_i, F_\bullet)=(\Dmod_X\otimes_{V_D^0\Dmod_X}(\cV\otimes \cL_U^{i})^{(-D_J)^{\gee}+(0\cdot D_{I\setminus J})^{\lee}},F_\bullet),$$ 
that is, $(\cM_i,F_\bullet)$ underlies the the twisted intermediate extension ${j_{1_J}}_*{j_{2_J}}_!V^i_{\cL_U}$.
\end{lemma}
By the filtered comparison in Proposition \ref{prop:maincomp}, we immediately obtain the $E_1$ degeneracy of the logarithmic Hodge-to-de-Rham spectral sequence.
\begin{cor}\label{cor:logHDR}
In the situation of Lemma \ref{lem:H-DR}, for each $J\subseteq I$ and $i$, we have that the logarithmic Hodge-to-de-Rham spectral sequence
\[E_1^{p,q}=\mathbb H^{p+q}\big(X,\gr^F_{-q}\DR_D\big({(\cV\otimes{\cL^i_U}})^{(0\cdot D_J)^{\gee}+D_{I\setminus J}^{\lee}}\big)\big)\]
\[\Rightarrow \mathbb H^{p+q}\big(X,\DR_D\big({(\cV\otimes{\cL^i_U}})^{(0\cdot D_J)^{\gee}+D_{I\setminus J}^{\lee}}\big)\big)\]
degenerates at $E_1$
\end{cor}

When extremely $J=I$, the $\Dmod$-module underlying the twist extension
$${j_{1_J}}_*{j_{2_J}}_!V^i_{\cL_U}=j_*V^i_{\cL_U}$$ 
is the Deligne meromorphic extension $(\cV\otimes \cL^i_U)(*D)$, thanks to Proposition \ref{prop:shiftind} (3). In this case the smallest term in its Hodge filtration satisfying the following local vanishing under birational morphisms, which is a natural generalization of Grauert-Riemenschneider vanishing and relative Kawamata-Viehweg vanishing in the log canonical case; see also the proof of Theorem \ref{thm:lcsur}. 
\begin{thm}\label{prop:locvanlc}
Let $\mu: X\to Y$ be a projective birational morphism between complex varieties with $X$ smooth and a reduced divisor $E$ on $Y$ so that  $D=f^{-1}E$ is normal crossing. Assume that $V=(\cV, F_\bullet, \VV_\R)$ is a polarizable VHS on $U=X\setminus D$ and $\cL_U$ a torsion line bundle on $U$ so that $\cL_U^m\simeq\cO_U$. Then we have
\[R^j\mu_*(S^{\gee -D}(V^i_{\cL_U})\otimes\omega_X)=0, \textup{ for } j>0.\]
\end{thm}
\begin{proof}
By functoriality of direct images for (mixed) Hodge modules, we know 
\[\mu_*(j_*V^i_{\cL_U})\simeq j_{E*}({\tilde\mu}_*(V^i_{\cL_U})),\]
where $\tilde\mu=\mu|_U$ and $j_E\colon Y\setminus E\to Y$. By Saito's direct image theorem for pure Hodge modules (see \cite[Thm. (0.3)]{Sai90a} for the real case), we have non-canonical decomposition
\[{\tilde\mu}_*(V^i_{\cL_U})=\bigoplus \cH^i{\tilde\mu}_*(V^i_{\cL_U})[-i].\]  
Moreover, the strict support decomposition \cite[Prop. (1.5)]{Sai90a} gives a direct sum
\[\cH^i{\tilde\mu}_*(V^i_{\cL_U})\simeq\bigoplus_ZM^i_Z\]
over closed subvarieties of $Y\setminus E$. 

On the other hand, since $j_E$ is an affine morphism, it preserves exactness. But the Hodge-module direct image is compatible with the direct image for the underlying perverse structure (see \cite{Sai90b}). Hence, we know 
\[\cH^i\mu_*(j_*V^i_{\cL_U})\simeq j_{E*}\big(\cH^i{\tilde\mu}_*(V^i_{\cL_U})\big)\simeq\bigoplus_Zj_{E*}M^i_Z.\]

By \cite[Thm. 2.14]{Sai90b} (which holds for real mixed Hodge modules as well), we see the complex $\mu_*(j_*V^i_{\cL_U})$ is strict, which in particular means 
\begin{equation}\label{eq:inj222}
R^i\mu_*(F_{p(V)}(\cV\otimes \cL^i_U)(*D)\otimes \omega_X)\simeq F_{p(V)-n}\cH^{i}\mu_+\big((\cV\otimes \cL^i_U)(*D)\big),
\end{equation}
where $n$ is the dimension of $X$, $\mu_+$ is the direct-image functor for right $\Dmod$-modules and $p(V)\colon=\min\{p| F_p\cV\neq 0\}$. Furthermore, by \cite[Prop. 2.6]{SaitoKC}, we know 
\begin{equation}\label{eq:id333}
F_{p(V)-n}M_Z^i=0,
\end{equation} 
if $Z\neq Y\setminus E$. But since $\mu$ is birational, clearly $\cH^i{\tilde\mu}_*(V^i_{\cL_U})$ is supported on a proper subscheme of $Y\setminus E$ when $i\neq 0$. Since the filtration on $(\cV\otimes \cL^i_U)(*D)$ is defined by convolution, we know
\[F_{p(V)}(\cV\otimes \cL^i_U)(*D)\otimes \omega_X=S^{\gee -D}(V^i_{\cL_U})\otimes\omega_X.\]
Therefore, the local vanishing follows from \eqref{eq:inj222} and \eqref{eq:id333}.
\end{proof}

\section{Proofs of main theorems}
In this section, we prove the main theorems in \S1.2.

\noindent
{\bf Proof of Theorem \ref{thm:maininj}.} We prove the first statement. We write 
\[D'=\sum_{i\in I} a_i D_i\]
and set 
\[J=\{i\in I| a_i\neq0\}.\]
First, we take $\cL_U=\cL|_U$, and consider the first twisted VHS, $V_{\cL_U}$. Observe that the residue map along $D_i$ of $\cL$ is exactly ${a_i}/{N}$. In this case, the condition $\frac{1}{N}D'
\le D_u$ implies that 
$$(\cV\otimes \cL_U)^{(-D_J)^\gee+(0\cdot D_{I\setminus J})^\lee}=\cV^{\lee 0}\otimes \cL.$$ 
Then by construction, we know $(\Dmod_X\otimes_{V^0_D\Dmod_X}(\cV^{\lee 0}\otimes \cL), F_\bullet)$ underlies the twisted intermediate extension ${j_{1_J}}_*{j_{2_J}}_!V_{\cL_U}=j_!V_{\cL_U}$.

On the other hand, applying Proposition \ref{prop:shiftind} (5), we have identity
\begin{equation}\label{eq:id9999}
\Dmod_X\otimes_{V^0_D\Dmod_X}(\cV^{\lee 0}\otimes \cL)=\Dmod_X\otimes_{V^0_D\Dmod_X}(\cV^{\lee 0}\otimes \cL(E))
\end{equation}
for every effective divisor $E$ supported on $\supp (D')$. Since the Filtration on $\Dmod_X\otimes_{V^0_D\Dmod_X}(\cV^{\lee 0}\otimes \cL)$ is given by convolution,  the identity \eqref{eq:id9999} gives inclusions
\[S^{\lee 0}(V)\otimes \cL\subseteq S^{\lee 0}(V)\otimes \cL(E)\subseteq \Dmod_X\otimes_{V^0_D\Dmod_X}(\cV^{\lee 0}\otimes \cL),\]
and hence inclusions of complexes
\[S^{\lee 0}\otimes\omega_X\otimes \cL\subseteq S^{\lee 0}\otimes\omega\otimes \cL(E)\hookrightarrow \DR(\Dmod_X\otimes_{V^0_D\Dmod_X}(\cV^{\lee 0}\otimes \cL)).\]
Taking cohomology, since 
\[S^{\lee 0}\otimes\omega_X\otimes \cL=F_{\textup{smallest}}\DR(\Dmod_X\otimes_{V^0_D\Dmod_X}(\cV^{\lee 0}\otimes \cL)),\] 
the $E_1$-degeneration in Lemma \ref{lem:H-DR} gives the injection 
\[H^i(X, S^{\lee0}(V)\otimes\omega_X\otimes \cL)\to \mathbb H^i(X, \DR(\Dmod_X\otimes_{V^0_D\Dmod_X}(\cV^{\lee 0}\otimes \cL)))\]
factoring through $H^i(X, S^{\lee0}(V)\otimes\omega_X\otimes \cL(E))$. Therefore, we conclude that the natural morphism
\[H^i(X, S^{\lee0}(V)\otimes\omega_X\otimes \cL)\to H^i(X, S^{\lee0}(V)\otimes\omega_X\otimes \cL(E))\]
is also injective.

After taking $J$ empty and considering $V^{-1}_{\cL_U}$, the second statement follows from exactly the same arguments.
\qed
\\
\\
\noindent
{\bf Proof of Theorem \ref{thm:inj}.}  We first write
\[L-B\sim_\R A+F+E\]
as a sum of $\R$-divisors, where $A$ is ample, $F$ effective,  $A+E$ big and nef and $B$ supported on $D$. Since ample cone is open, by perturbing coefficients, we can assume  $A$, $E$ and $B$ are all $\Q$-divisor. By the upper semi-continuity of the upper Deligne lattice, we know that $S^{\lee B}(V)$ remains the same after perturbation. Here we use $\sim$ to denote linear equivalence.

We then take a log resolution of $(X, E+D)$, $\mu\colon X'\to X$ with centers inside the singularities of $E+F+D$. We can further assume that 
\[\mu^*(L-B)\sim_\Q \frac{1}{N} G\]
where $G$ is an effective divisor so that $G+\mu^*D$ is normal crossing for $N\gg 0$ and sufficient divisible and 
$\supp(\mu^*E)\subseteq \supp (G)$.

Now we take $\cL=\cO_{X'}(\mu^*L)$ and consider . Then we have 
\[\cL^N\sim \cO_{X'}(\mu^*(N\cdot B)+G).\]
Take $U=X'\setminus \supp(\mu^*D+G)$ and $\cL_U= \cL|_U$.
By considering the eigenvalues of the residues, one can see 
\[S^{\lee \mu^*B}(V)\otimes \cL=S^{\lee 0}(V_{\cL_U}).\]
By our choice of $N$, the eigenvalues of the residue along each irreducible component of $\supp (G)$ is not 0. Hence, the proof of Theorem \ref{thm:maininj} yieldes the injection of cohomology groups
\[H^i\big(X',S^{\lee \mu^*B}(V)\otimes\omega_{X'}(\mu^*L)\big)\to H^i\big(X',S^{\lee \mu^*B}(V)\otimes\omega_{X'}(\mu^*L+H))\big),\]
when $H$ is an effective divisor supported on $\supp(G)$. In particular, the natural morphism 
\begin{equation}\label{eq:injup}
H^i\big(X',S^{\lee \mu^*B}(V)\otimes\omega_{X'}(\mu^*L)\big)\to H^i\big(X',S^{\lee \mu^*B}(V)\otimes\omega_{X'}(\mu^*(L+E)))\big)
\end{equation}
is injective. Choosing $H$ sufficiently ample, by Serre vanishing we see 
\[H^i\big(X',S^{\lee \mu^*B}(V)\otimes\omega_{X'}(\mu^*L)\big)=0, \textup{ } i>0.\]
Picking $E$ sufficiently ample, the standard Leray-spectral-sequence argument shows that
\[R^j\mu_*(S^{\lee \mu^*B}(V)\otimes\omega_{X'})=0\]
for $j>0$. Pushing forward \eqref{eq:injup}, the proof is accomplished by Lemma \ref{lem:birSV}. 
\qed
\\
\\
\noindent
{\bf Proof of Theorem \ref{thm:NakanoVan}.} We prove the first statement. \\
\noindent
{\bf Step 1. }By definition of $\R$-linear equivalence, after perturbing coefficients we can assume the coefficients $a_i$ of $\Delta$ are all rational, and
\[L\sim_\Q -\Delta +\frac{1}{k}S\]
where $S$ is a semi-ample divisor. Then we pick general  $A\in \abs{\cO_X(mS)}$ for some sufficiently large and divisible $m>0$ so that $km\Delta$ is integral, $A$ smooth and $D=E +A$ normal crossing. Renaming the components of $D$, we set $D=\sum_{i\in I}D_i$, and $\Delta=\sum_i a_i D_i$.

\noindent
{\bf Step 2. }Consider $\cL=\cO_X(L-A)$ and $\cL_U=\cL|_U$ and the trivial VHS $V=(\cO_U, F_\bullet, \R_U)$ over $U$, where $U=X\setminus D$. Set 
$$J=I\setminus\{i\in I | \textup{ }a_i=1,\textup{ }a_i \textup{ coefficients of }\Delta\}.$$  Since 
$$L-A\sim_\Q-\Delta-\frac{km-1}{km}A,$$
we see that 
\[\cL(D)=(\cL_U)^{(-D_{J})^{\gee}+(0\cdot D_{I\setminus J})^\lee}\]
and 
\[{j_{1_J}}_*{j_{2_J}}_!V_{\cL_U}=(\Dmod_X\otimes_{V^0_D\Dmod_X}\cL(D), F_\bullet,  {Rj_{1_J}}_*{j_{2_J}}_!L_U).\]
Then we apply Lemma \ref{lem:H-DR}, we get $E_1$ degeneracy of the Hodge-to-de-Rham spectral sequence for ${j_{1_J}}_*{j_{2_J}}_!V_{\cL_U}$. Since in this case the filtration on $V_{\cL_U}$ is trivial, Corollary \ref{cor:logHDR} yields that the logarithmic Hodge-to-de-Rham spectral sequence 
\[H^{q}(X, \Omega^p(\log D)\otimes \cL)\Rightarrow H^{p+q}(X, {Rj_{1_J}}_*{j_{2_J}}_!L_{U})=H^{p+q}(X\setminus D_J, {j_{2_J}}_!L_U)\]
degenerate at the $E_1$ term. But the affineness assumption implies  that $X\setminus D_J$ is affine. Hence, by Artin-Grothendieck vanishing (cf. \cite[Thm. 3.1.13]{Laz1}), we know 
\[H^{p+q}(X\setminus D_J, {j_{2_J}}_!L_U)=0, \textup{ for } p+q>n.\] 
Therefore, the $E_1$ degeneration of the logarithmic Hodge-to-de-Rham spectral sequence gives us 
\begin{equation}\label{eq:van}
H^{q}(X, \Omega^p(\log D)\otimes \cL)=0,\textup{ for } p+q>n.
\end{equation}

We then prove the vanishing 
\[H^{q}(X, \Omega^p(\log E)\otimes \cL)=0,\textup{ for } p+q>n\]
by induction. If $p=0$, then $q>n$ and the assertion we need to prove is trivial. Suppose now that $q>0$ and consider the short exact sequence
\[0\to \Omega^p(\log E+A)(-A)\to \Omega^p(\log E)\to \Omega^p(\log E|_A)\to 0.\]
By tensoring with $\cO_X(L)$ and taking the long exact sequence in cohomology, we have a long exact sequence
\[H^{q}(X, \Omega^p(\log D)\otimes \cL)\to H^{q}(X, \Omega^p(\log E)(L))\to H^{q}(A, \Omega^p(\log E|_Q)(L|_A)).\]
When $p+q>n$, the first term vanishes by \eqref{eq:van}, while the third term vanishes by the inductive assumption. We thus obtain the vanishing of the second term and the proof of the first statement is accomplished.

To prove the second statement, after perturbing coefficients we can assume
\[L\sim_\Q -\Delta +\frac{1}{k}T\]
where $T$ is an ample divisor. Similar to the proof of the first statement, we pick general  $H\in \abs{\cO_X(mT)}$ for some sufficiently large and divisible $m>0$ so that $km\Delta$ is integral, $H$ smooth and $D=E +H=\sum_{i\in I}D_i$ normal crossing. Then we consider $\cL=\cO_X(L-H)$ and $\cL_U=\cL|_U$ and the trivial VHS $V=(\cO_U, F_\bullet, \R_U)$ over $U$, where $U=X\setminus D$. In this case, we set 
$$J=I\setminus \{i\in I | \textup{ }a_i=1,\textup{ }a_i \textup{ coefficients of }\Delta+\frac{1}{mk}A\}.$$  Since $H$ is very ample, we see that $X\setminus D_J$ is affine. Now we follow Step 2 of the proof of the first statement, and we obtain the desired vanishing. 
\qed

\section{Applications}

\subsection{A Fujita-type freeness theorem}
In this section, we prove Fujita-type global generation for $S^{\lee0}(V)\otimes \omega_X$, using Theorem \ref{thm:mainvan}.

Assume that $\Delta^n$ is a polydisk with holomorphic coordinates$(z_1,\dots,z_n)$ and $D_i$ is the divisor defined by $z_i=0$ and $D=\sum_{1\le i\le k} D_i$ for $k\le n$. We first collect the following lemma about local indecomposable decompositions of VHS, whose proof is immediate; see also \cite[Proof of Lemma 3.2]{PTW}  and \cite[(3.10.6)]{Sai90b}.
\begin{lemma}[Local indecomposable decomposition]\label{lem:lid}
Let $V$ be a polarizable real VHS on $\Delta^n\setminus D$. Then for every $\R$-divisor $B=\sum_{1\le i\le n} t_iD_i$, the Deligne lattice $V^B$ has an indecomposable decomposition as polarizable VHS
\[V^{\lee B}=\bigoplus_{\alpha=(\alpha_1,\dots,\alpha_n)} V^{\lee B}_\alpha\]
where $V^{\lee B}_\alpha=(\cV^{\lee B}_\alpha, F_{\alpha,\bullet}^{\lee B},\VV_{\R, e^{-2\pi\sqrt{-1}\alpha}})$ for some $t_i-1<\alpha_i\le t_i$. The decomposition is induced by the eigenspace decomposition 
$$\VV_\R=\bigoplus_{\alpha=(\alpha_1,\dots,\alpha_n)}\VV_{\R, e^{-2\pi\sqrt{-1}\alpha}},$$ where $\VV_{\R, e^{-2\pi\sqrt{-1}\alpha}}$ is the (simutaneous) generalized eigenspace of $\VV_\R$ with eigenvalues $e^{-2\pi\sqrt{-1}\alpha_i}$ with respect to monodromy actions around all $z_i$. In particular, we have the decomposition
\[S^{\lee B}(V)=\bigoplus_{\alpha=(\alpha_1,\dots,\alpha_n)}S^{\lee B}(V)_\alpha.\]
\end{lemma}

The following theorem is a Hodge-module generalization of the relative Fujita-type freeness result in \cite{Kaw02}, where the global generation for higher direct images of dualizing sheaves in the normal crossing case follows from an alternative point of view using Hodge metrics and Kawamata coverings.
\begin{thm}\label{thm:fgg}
Let $X$ be a smooth projective variety of dimension $n$ with a simple normal crossing divisor $D$ and $V$ a polarizable real VHS on $X\setminus D$, and let $L$ be an ample divisor on $X$ and a point $x\in X$. Assume that for every klt pair $(X, B_0)$, there exists an effective $\Q$-divisor $B$ on $X$ satisfying the following conditions:
\begin{enumerate} [label=(\roman*)]
\item{$B\equiv \lambda L$ (numerical equivalence) for some $0<\lambda<1$;}
\item{$(X, B+B_0)$ is log canonical at $x$;}
\item{$\{x\}$ is a log canonical center of $(X, B+B_0)$.}
\end{enumerate}
Then the natural morphism 
\[H^0(X, S^{\lee 0}(V)\otimes \omega_X(L))\rarr S^{\lee 0}(V)\otimes \omega_X(L)|_{\{x\}}\]
is surjective. 
\end{thm}

\begin{proof}
Let $\{s_1,...,s_k\}$ be a basis of $S^{\lee0}(V)_\alpha$ on a polydisk neighborhood $W$ around $x$, for some index $\alpha=(\alpha_1,\dots,\alpha_n)$ as in Lemma \ref{lem:lid}. After considering all direct summands $S^{\lee0}(V)_\alpha$ locally around $x$, it is sufficient to prove that the image of the morphism
\[H^0(X, S^{\lee 0}(V)\otimes \omega_X(L))\rarr S^{\lee 0}(V)\otimes \omega_X(L)|_{\{x\}}\]
contains $s_j\otimes t|_{\{x\}}$ for any $j$, where $t$ is a local generator of $\omega_X(L)$.

We set $B_{s_j}=\sum_i\alpha_iD_i$. Since $-1<\alpha_i\le0$ for all $i$, we have that $(X, -B_{s_j})$ is a klt pair. After positively perturbing coefficients of $B_{s_j}$ (if necessary), we can assume it is a $\Q$-divisor. By the assumption of the theorem, there exists an effective $\Q$-divisor $B$ such that 
\begin{enumerate} [label=(\roman*)]
\item{$B\equiv \lambda L$ for some $0<\lambda<1$;}
\item{$(X, B-B_{s_j})$ is log canonical at $x$;}
\item{$\{x\}$ is a log canonical center of $(X, B-B_{s_j})$.}
\end{enumerate}
Take a general element $B'$ of $\abs{mL}$ for $m\gg 0$ passing through $x$. After replacing $B$ by 
\[(1-\epsilon_1)B+\epsilon_2B'\]
for some suitable $0< \epsilon_i\ll \frac{1}{m}$, we can assume that $\{x\}$ is the only log canonical center of $(X, B-B_{s_j})$ passing through $x$. Let $\mu: X'\rarr X$ be a log resolution of $(X,D+\textup{supp}(B))$. Then 
\[\mu^* (B-B_{s_j})=K_{X'/X}+E+F^{+}-F^{-}\]
where $K_{X'/X}$ the relative canonical divisor and $E$ is the only log canonical place over $x$ (make a further perturbation of $B$ if there are more than one log canonical place over $x$), $F^{+}$ is effective satisfying $\floor {F^{+}}|_{\mu^{-1}W}= 0$ and $F^{-}$ is effective. Then we have a short exact sequence
\[0\rarr S^{\lee\mu^*B_{s_j}+E+F^{+}-F^{-}}(V)\otimes \mu^*\omega_X(L)\rarr S^{\lee\mu^*B_{s_j}+F^{+}-F^{-}}(V)\otimes \mu^*\omega_X(L)\to\]
\[\rarr S^{\lee\mu^*B_{s_j}+F^{+}-F^{-}}(V)\otimes \mu^*\omega_X(L)|_E\rarr 0.\]
By Lemma \ref{lem:lid} and local computation of residues, one sees that $\mu^*(S^{\lee0}(V)_\alpha)$ is a direct summand of $S^{\lee\mu^*B_{s_j}}(V)$ over $\mu^{-1}W$. Since $\floor {F^{+}}|_{\mu^{-1}W}= 0$ and we can make $\mu^*B_{s_j}$very close to $\mu^*(\sum\alpha_i D_i)$, we also see that  $\mu^*(S^{\lee0}(V)_\alpha)$ is a direct summand of $S^{\lee \mu^*B_{s_j}+F^{+}}(V)$ over $\mu^{-1}W$ (notice that the $S^{\lee0}(V)_\alpha$ and its pullback jump integrally to the next ones). Therefore, since $F^{-}$ is effective, we have 
\[\mu^*(s_j\otimes t)\in \Gamma(S^{\lee\mu^*B_{s_j}+F^{+}-F^{-}}(V)\otimes \mu^*\omega_X(L), \mu^{-1}W).\]
Moreover, since $E$ is the only log canonical place over $x$ and $E\subseteq \mu^{-1}W$, we have 
\[0\neq\mu^*(s_j\otimes t)|_E=\mu^*(s_j\otimes t|_{\{x\}}).\]

But we also know
\[S^{\lee\mu^*B_{s_j}+E+F^{+}-F^{-}}(V)\otimes \mu^*\omega_X(L)=S^{\lee \mu^*B}(V)\otimes\omega_{X'}(\mu^*L).\]
Hence, the vanishing in Theorem \ref{thm:mainvan} enables us to lift $\mu^*(s_j\otimes t)|_E$ to a section in 
\[H^0(X', S^{\lee\mu^*B-E}(V)\otimes\omega_{X'}(\mu^*L)).\]
Since $B$ is effective and $E$ is exceptional, the proof is accomplished by the inclusions
\[H^0(X', S(V)^{\lee \mu^*B-E}\otimes\omega_{X'}(\mu^*L))\subseteq H^0(X', S(V)^{\lee 0}(E)\otimes\omega_{X'}(\mu^*L))\]
and
\[H^0(X', S(V)^{\lee 0}(E)\otimes\omega_{X'}(\mu^*L))\subseteq H^0(X, S(V)^{\lee 0}\otimes\omega_{X}(L)).\]
\end{proof}

\subsection{Surjectivity for log canonical pairs}
As an application of Item \eqref{item:inj2}  in Theorem \ref{thm:maininj} (or more precisely Esnault-Viehweg injectivity as in Corollary \ref{cor:injtrivial} (2)) and the local vanishing in Theorem \ref{prop:locvanlc} for the trivial VHS, we obtain the following surjectivity for log canonical pairs, which, we believe, is known to experts. 

Indeed, it is conjectured to be true more generally when $Z$ is only normal and $(Z, \Delta)$ is log canonical in \cite[Ch. 12]{kolSh} and can be deduced by using the main theorem in  \cite{KK} that log canonical singularities are Du Bois. 

We want to give it a proof using Hodge modules and the injectivity obtained in this article under the extra condition that $Z$ is Gorenstein. To our knowledge, the local vanishing in Theorem \ref{prop:locvanlc} even for the trivial VHS is not known before in literature. In other words, the following surjectivity cannot be proved by only using the injectivtity of Esnault and Viehweg besides the normal crossing case.
\begin{thm}\label{thm:lcsur}
Let $Z$ be a projective normal varieties with Gorenstein singularities and let $L$ be a Cartier divisor. Assume that $(Z, \Delta)$ is a log canonical pair with $\Delta$ an $\Q$-Cartier divisor.  If $L-\Delta$ is semi-ample, then for every effective Cartier divisor $E$ supported on $\supp \Delta$ the natural morphism 
\[H^i(Z, \cO_Z(-L-E))\to H^i(Z, \cO_Z(-L))\]
is surjective, for every $i$. 
\end{thm} 
\begin{proof}
Since $Z$ is assumed to be Gorenstein, by Serre duality, it is enough to obtain the injectivity of the nature morphism
\begin{equation}\label{eq:inj555}
H^i(Z, \omega_Z(L))\to H^i(Z, \omega_Z(L+E)).
\end{equation}
To this end, we take $\mu\colon X \to  Z$ a log resolution of the pair $(Z, \Delta)$ and assume 
\[K_{X}+\Delta_X\sim_\R \mu^*(K_Z+\Delta),\]
where $\Delta_X$ is normal crossing and $K_X$ and $K_Z$ are the canonical divisors. 
Since $(Z, \Delta)$ is log canonical, we know $\Delta_X$ is also log canonical.  Since $L-\Delta$ is semi-ample, we can assume that (perturb the coefficients of $\Delta_X$ if needed)
\[\mu^*L-K_{X/Z}\sim_\Q \Delta_X+\frac{1}{k}A\]
where $A$ is a smooth divisor so that $\Delta_X+A$ is normal crossing and $k>1$ a positive integer. Write $\Delta_X=\Delta_X^{+}-\Delta_X^{-}$ where $\Delta_X^{+}$ and $\Delta_X^{-}$ are effective with no common components. Then we have
\[\mu^*L-K_{X/Z}+\ceil{\Delta_X^{-}}\sim_\Q \Delta_X+\ceil{\Delta_X^{-}}+\frac{1}{k}A\]
and $\Delta_X+\ceil{\Delta_X^{-}}+\frac{1}{k}A$ is positive and log canonical. We also write $D=\supp(\Delta_X+A)$ for simplicity. By adding exceptional divisors (with coefficients 1) not supported on $D$ to $\Delta_X^{-}$, we can assume that $D$ contains all the exceptional divisor of $\mu$.  

We now take $V$ the trivial VHS on $U=X\setminus D$ and $\cL=\cO_X(D-\mu^*L+K_{X/Z}-\ceil{\Delta_X^{-}})$ and $\cL_U=\cL|_U$. 
By Corollary \ref{cor:injtrivial} (2), we thus obtain that the natural map
\begin{equation}\label{eq:inj444}
H^i(X, \omega_X\otimes \cL^{-1}(D))\to H^i(X, \omega_X\otimes\cL^{-1}(\pi^*E))
\end{equation} 
is injective. By construction, it is obvious that
\[\mu_*(\omega_X\otimes \cL^{-1}(D))=\omega_X(L)\]
and 
\begin{equation}\label{eq:inj3333}
S^{\gee -D}(V^{-1}_{\cL_U})=\cL^{-1}(D).
\end{equation} 
Pushing-forward the natural map \eqref{eq:inj444}, the local vanishing in Theorem \ref{prop:locvanlc} for the trivial VHS and identification \eqref{eq:inj3333} yield the desired injectivity in \eqref{eq:inj555}.

\end{proof}
\section*{References}
\bibliography{bibliography/general}

\end{document}